\documentclass{amsart}
\usepackage{amsmath,amsthm,amssymb,IMjournal}
\usepackage{amsfonts,tracefnt,amsgen,amsmath,amsthm,amscd,amssymb,stmaryrd, commath, mathrsfs, enumerate, stmaryrd}
\usepackage{commath}

\begin{document}
\newtheorem{thm}{Theorem}[section]
\newtheorem{lem}[thm]{Lemma}
\newtheorem{pro}[thm]{Proposition}
\newtheorem{conj}[thm]{Conjecture}
\newtheorem{cor}[thm]{Corollary}
\newtheorem{prob}[thm]{Problem}
\newcommand{\DD}{\mathbb{D}}
\theoremstyle{definition}
\newtheorem{defi}[thm]{Definition}
\newtheorem{ex}[thm]{Example}
\newtheorem{question}[thm]{Question}
\newtheorem{remark}[thm]{Remark}
\newcommand{\CC}{{\mathbb C}}
\newcommand{\RR}{{\mathbb R}}

\newcommand{\vrho}{\varrho}
\newcommand{\D}{\mathbb{D}}
\newcommand{\TT}{\mathbb{T}}
\newcommand{\C}{\mathbb{C}}
\newcommand{\p}{{p(\cdot)}}
\newcommand{\q}{{q(\cdot)}}
\newcommand{\tcr}[1]{\textcolor{red}{#1}}
\newcommand{\tcb}[1]{\textcolor{blue}{#1}}
\newcommand{\pprime}{{p'(\cdot)}}

\newcommand{\essi}{\operatornamewithlimits{ess\,inf}}
\newcommand{\esss}{\operatornamewithlimits{ess\,sup}}
\providecommand{\norm}[1]{\lVert#1\rVert}
\newcommand{\LS}{\ensuremath{L^{p(\cdot)}(\Omega,\mu)}}

\numberwithin{equation}{section}

\title{Analytic Variable Exponent Hardy Spaces}

\author{\|Gerardo R. |Chac\'on|, Washington D.C.,
        \|Gerardo A. |Chac\'on|, Bogot\'a}



\abstract 
   We introduce a variable exponent version of the Hardy space of analytic functions on the unit disk, we show some properties of the space, and give an example of a variable exponent $\p$ that satisfies the $\log$-H\"older condition such that $H^\p\neq H^q$ for any constant exponent $1<q<\infty$. We also consider the variable exponent version of the Hardy space on the upper-half plane.
  \endabstract



\thanks
   The research has been partially supported by research project 2015004: ``Contribuciones a la teor\'ia de operadores en espacios de funciones anal\'iticas''. Universidad Antonio Nari\~no
\endthanks

\section{Introduction}
Variable Exponent Lebesgue spaces  are a generalization of classical Lebesgue spaces $L^p$  in which the exponent $p$ is a measurable function. Such spaces where introduced by  Orlicz \cite{orlicz} and developed by Kov\'a{\v{c}}ik and R\'akosn\'{i}k \cite{kovrak}.    Although such spaces have received a considerable amount of attention, little is known about their analytic version. 

Recently, the research subject has received increasing interest and some progress has been made. For example,  in \cite{KS} a version of $BMO$ spaces with variable exponents is considered. Bergman spaces with variable exponents have been studied in \cite{charaf2014,  charaf2016b, charaf2016} and a different approach has been taken in \cite{KaSa16b} and \cite{KaSa16}, much of the research done in the area assume the $\log$-H\"{o}lder condition on the exponent, which is a growth condition that usually guarantees the boundedness of a Hardy-Littlewood maximal operator in a related space.  One case of function spaces on unbounded domains have been studied in \cite{MPT12}. The theory of  Orlicz spaces of analytic functions has also received recent interest, see for example \cite{LLQR10} and  \cite{LHQR13}. Little is known about a variable exponent version of Hardy spaces on the unit disk. A first approach was taken in  \cite{KP1} and \cite{KP2}. In this article, we introduce a version of variable exponent Hardy spaces on the unit disk and on the upper half-plane, we prove a Poisson-type representation in both cases, and show an estimate for the norm of the reproducing kernels. We also show an an example of a variable exponent $\p$  such that $H^\p\neq H^q$ for any constant exponent $1<q<\infty$. Most of the result presented in this article are analogous to the classical case, we include detailed proofs for the sake of completeness. 

The article is distributed as follows. In the next section, we consider the case of the unit disk by first introducing the harmonic version of the spaces, and showing a representation in terms of the Poisson kernel and the boundary values of the functions. We also prove an estimate for the norm of the reproducing kernels on the variable exponent Hardy spaces. In Section 3 we show through an example that variable exponent Hardy spaces do not coincide with the classical counterparts, even under the usual regularity conditions of the exponent (the Log-H\"older condition). Finally, section 4 is devoted to  study variable exponent Hardy spaces on the upper half-plane. 

While preparing this article, the authors found the preprint \cite{Fer} in which similar questions are studied. However, the techniques and results obtained here are different and have little overlapping.

\section{Variable exponent Hardy spaces in the unit disk}


We begin this section by introducing the preliminary concepts of variable exponent spaces.

\begin{defi}
Let $p:[0,2\pi]\to [1,\infty)$ be a measurable, essentially bounded function such that $p(0)=p(2\pi)$. The space $L^\p(\TT)$ is defined as the space of all functions $f:\TT\to\C$ such that \[\rho_\p(f)=\int_0^{2\pi} |f(e^{i\theta})|^{p(\theta)}\dif \theta<\infty.\] Such space is a Banach space with respect to the norm: \[\norm{f}_{L^\p(\TT)}= \inf \left\{ \lambda > 0:
\rho_\p\left(\frac{f}{\lambda}\right)\leqslant 1 \right\}.
\]
\end{defi}

We will use a regularity condition on the exponent $\p$ which is common in the study of variable exponent spaces: 

\begin{defi}\label{def:log-Holder}
	A function $p:[0,2\pi] \to [1,\infty)$ is said to be \textit{$\log$-H\"older continuous}  or to satisfy the \textit{Dini-Lipschitz condition} on $[0,2\pi]$ if there exists a positive constant $C_{\log}$ such that
\begin{equation*}
	|p(x)-p(y)|\leqslant \frac{C_{\log}}{\log\left( 1/|x-y| \right)},
	\label{eq:local}
\end{equation*}
for all $x, y \in [0,2\pi]$. We will denote as $p^+ =    \esss_{x\in [0,2\pi]}p(x)$ and $ p^-=\essi_{x\in [0,2\pi]}p(x)$.

\end{defi}

\begin{defi}
For each $z$ in the unit disk $\D$, the Poisson kernel $P(z,\zeta)$ is defined as \[P(z,\zeta)= \frac{1-|z|^2}{|z-\zeta|^2}\] and the Poisson transform of a function $f\in L^\p(\TT)$ is defined as \[Pf(z)=\int_\TT P(z,\zeta) f(\zeta)\dif m(\zeta).\]
\end{defi}

We will use the following result from \cite{Shara}:

\begin{thm}\label{thm:Shara_Poisson}
Suppose that  $p:[0,2\pi] \to [1,\infty)$ is $\log$-H\"older continuous. Given $0<r<1$, and $f\in L^\p(\TT)$,  the linear operators $P_r:L^\p(\TT)\to L^\p(\TT)$ defined as \[P_rf(\zeta)=Pf(r\zeta)\] are uniformly bounded on $L^\p(\TT)$, $\norm{P_r(f)}_{L^\p(\TT)}\lesssim \norm{f}_{L^\p(\TT)}$, and for every $f\in L^\p(\TT)$, \[\norm{f-P_rf}_\p\to 0,\quad\text{as }r\to 1^-\]
\end{thm}
 
 We are now ready to define the harmonic Hardy spaces with variable exponents.
 
 \begin{defi}
 Let  $f:\D\to\C$ and for $0<r<1$, define $f_r:\TT\to\C$ as $f_r(\zeta)=f(r\zeta)$. Given $p:[0,2\pi] \to [1,\infty)$, the {\em variable exponent harmonic Hardy space} $h^\p(\D)$ is defined as the space of harmonic functions $f:\D\to\C$ such that \[\norm{f}_{h^\p(\D)}=\sup_{0\leq r<1}\norm{f_r}_{L^\p(\TT)}<\infty.\]
 \end{defi}

Notice that, since $Pf$ is harmonic for $f\in L^\p(\TT)$, then Theorem \ref{thm:Shara_Poisson} shows that the Poisson transform $P:L^\p(\TT)\to h^\p(\D)$ is bounded. Moreover, we have the following theorem analogous to the constant exponent context. We follow the ideas in \cite{Mas}.

\begin{thm}\label{thm:Poisson}
Suppose that  $p:[0,2\pi] \to [1,\infty)$ is $\log$-H\"older continuous. Then for every $U\in h^\p(\D)$, there exists $u\in L^\p(\TT)$ such that $Pu=U$ and moreover, $\norm{U}_{h^\p(\D)}\sim \norm{u}_{L^\p(\TT)}$.
\end{thm}

\begin{proof}
Suppose that $U\in h^\p(\D)$ and for $n\geq 2$ define the dilations \[U_n(z)=U\left(\left(1-\frac{1}{n}\right)z\right).\] then $U_n$ is harmonic in a neighborhood of $\overline{\D}$ and consequently \[U_n(z)=\int_\TT P(z,\zeta)U_n(\zeta)\dif m(\zeta).\]

Now let $q:[0,2\pi] \to [1,\infty)$ be a measurable function such that $\frac{1}{p(t)}+\frac{1}{q(t)}=1$ for every $t\in[0,2\pi]$. Let's define the linear functionals $\Lambda_n:L^\q(\TT)\to\C$ as
\[\Lambda_n(f)=\int_\TT f(\zeta)U_n(\zeta)\dif m(\zeta).\] By H\"older's inequality, we have that $|\Lambda_n(f)|\lesssim \norm{U_n}_{L^\p(\TT)}\norm{f}_{L^\q(\TT)}$, consequently $\Lambda_n$ belongs to the dual space $(L^\q)^\ast$ and $\norm{\Lambda_n}\lesssim \norm{U_n}_{L^\p(\TT)}\leq \norm{U}_{h^\p(\D)}$.

We now use Banach-Alaoglu's Theorem to find $\Lambda\in \left(L^\q(\TT)\right)^\ast$ and a subsequence $\{\Lambda_{n_k}\}\subset \{\Lambda_n\}$  such that \[\Lambda_{n_k}(f)\to \Lambda(f),\qquad \text{as } k\to\infty\] for every $f\in L^\q(\TT)$ and $\norm{\Lambda}\lesssim \norm{U}_{h^\p(\TT)}$. 

The duality relation (see for example \cite{CUF}) on $L^\q(\TT)$ allow us to find a function $u\in L^\p(\TT)$ such that $\norm{u}_{L^\p(\TT)}\sim \norm{\Lambda}$ and for every $f\in L^\q(\TT)$, it holds that \[\Lambda(f)=\int_\TT f(\zeta)u(\zeta)\dif m(\zeta).\] In particular, 
$U_{n_k}(z) = \Lambda_{n_k}(P(z,\cdot))\to \Lambda(P(z,\cdot))$. Thus, \[U(z)=\int_\TT P(z,\zeta)u(\zeta)\dif m(\zeta).\] Moreover, putting the previous estimates of $\norm{U}_{h^\p(\D)}$ together with Theorem \ref{thm:Shara_Poisson} we get that \[\norm{U}_{h^\p(\D)}\sim \norm{u}_{L^\p(\TT)}.\]
\end{proof}

\begin{thm}\label{thm:subset}
Suppose that  $p:[0,2\pi] \to [1,\infty)$ is a $2\pi$-periodic function. Then \[h^{p^+}(\D)\subset h^\p(\D)\subset h^{p^-}(\D),\] and moreover, the inclusion is continuous.
\end{thm}

\begin{proof}
Suppose $f\in h^{p^+}(\D)$ and $\norm{f}_{h^{p^+}(\D)}\leq 1$, then by H\"older's inequality \[\rho_\p(f_r)=\int_0^{2\pi} |f(e^{i\theta})|^{p(\theta)}\dif \theta\lesssim \norm{|f_r|^\p}_{L^{p^+/\p}(\TT)}\norm{1}_{L^{p^+/(p^+-\p)}(\TT)}.\]
Now, since $\rho_{p^+/\p}(|f_r|^\p)=\rho_{p^+}(f_r)\leq \norm{f}_{h^{p^+}(\D)}\leq 1$, then there exists $C>0$ such that $\norm{|f_r|^\p}_{L^{p^+/\p}(\TT)}\leq C$ and therefore $\rho_\p(f_r)\lesssim 1$, so $f\in h^\p(\D)$.

For a general $f\in  h^{p^+}(\D)$, take $g=f\norm{f}_{h^{p^+}(\D)}^{-1}$ and apply the previous result to conclude that \[\norm{f}_{h^\p(\D)}\leq C\norm{f}_{h^{p^+}(\D)}.\]

Now suppose $f\in h^{\p}(\D)$ and $\norm{f}_{h^\p(\D)}\leq 1$, then for every $0\leq r<1$, it holds that $\rho_{\p}(f_r)\leq 1$. Now, by H\"older's inequality,
\[\rho_{p^-}(f_r)\lesssim \norm{|f_r|^{p^-}}_{L^{\p/p^-}(\TT)}\|1\|_{L^{\p/(\p-p^-)}(\TT)}.\] But $\rho_{\p/p^-}(|f_r|^{p^-})=\rho_\p(f_r)\leq 1$ and consequently, there exists a constant $C>0$ such that $\rho_{p^-}(f_r)\leq C$. Thus, $f\in h^{p^-}(\D)$. 

The result follows similarly as before for general $f\in h^{\p}(\D)$.
\end{proof}

As a consequence of Theorems \ref{thm:Poisson} and \ref{thm:subset}, we obtain the following result.

\begin{cor}
Suppose that  $p:[0,2\pi] \to [1,\infty)$ is a $2\pi$-periodic function. Then if $U\in h^\p(\D)$ and $U=Pu$ for $u\in L^\p(\TT)$, then for almost every $\theta \in  [0,2\pi],$\[\lim_{r\to1^-} U(re^{i\theta})=u(e^{i\theta}).\]
\end{cor}

\begin{proof}
This follows since $h^\p(\D)\subset h^1(\D)$.
\end{proof}

\begin{defi}
Suppose that  $p:[0,2\pi] \to [1,\infty)$ is a $2\pi$-periodic function. The variable exponent Hardy space $H^\p(\D)$ is defined as the space of analytic functions $f:\D\to\C$ such that $f\in h^\p(\D)$.
\end{defi}

In an analogous way to the classical setting (see for example \cite{Duren}) it is shown that $H^\p(\D)$ can be identified with the subspace of functions in $L^\p(\TT)$ whose negative Fourier coefficients are zero, and as such, $H^\p(\D)$ is a Banach space.

Recall that for functions in $f\in H^1(\D)$, we have the reproducing formula:

\begin{equation}\label{eq: reproducing}
f(z)=\int_\TT \frac{f(\zeta)}{1-\overline{\zeta}z}\dif m (\zeta).
\end{equation}

For each $z\in\D$, the functions $K_z:\D\to\CC$ defined as \[K_z(w)=\frac{1}{1-\overline{z}w},\] are called {\em reproducing kernels}. They are bounded on $\D$ and consequently belong to every space $H^\p(\DD)$. Moreover, as a consequence of the reproducing formula and Hanh-Banach theorem, the linear span of $\{K_z:z\in\D\}$ is dense in $H^\p(\D)$. Consequently, the set of polynomials is also dense in $H^\p(\D)$. 

\begin{thm}\label{thm:bounded_eval_funct}
Suppose that  $p:[0,2\pi] \to [1,\infty)$ is a $2\pi$-periodic function. For each $z\in \D$ consider the linear functional $\gamma_z:H^\p(\D)\to\C$ defined as $\gamma_z(f)=f(z)$. Then $\gamma_z$ is a bounded operator for every $z=|z|e^{i\theta}\in \D$ and \[\|\gamma_z\|\lesssim \frac{1}{(1-|z|)^{1/p(\theta)}}.\] Consequently, the convergence in the $H^\p(\D)$-norm implies the uniform convergence on compact subsets of $\D$.
\end{thm}

Before proving the Theorem, we will need the following technical Lemma, which is a version of a Forelli-Rudin inequality, adapted to our context.

\begin{lem}
Suppose that  $p:[0,2\pi] \to [1,\infty)$ is a $2\pi$-periodic, $\log$-H\"older continuous function and let $q:[0,2\pi] \to [1,\infty)$ be such that $\frac{1}{p(\theta)}+ \frac{1}{q(\theta)}=1$ for every $\theta\in [0,2\pi]$. Let $1/2<r<1$ and $z=|z|e^{i\theta}$. Define the function $\varphi:[0,2\pi]\to\mathbb{R}^+$ as \[\varphi(t)=\frac{(1-|z|)^{1/q(\theta)}}{|1-|z|re^{i(t-\theta)}|}.\] Then if $\varphi(t)>1$, it holds that \[\varphi(t)^{p(t)}\lesssim \varphi(t)^{p(\theta)}.\]
\end{lem}

\begin{proof}
We use the estimate \[1-|z|re^{i(t-\theta)}\sim |t-\theta|+1-r|z|\] that holds for every $1/2<r<1$. With this in hand, we get that for $1/2<r<1$, \[\varphi(t)\lesssim \frac{2\pi}{|t-\theta|}.\]

Now, notice that since $\p$ is $\log$-H\"older continuous, then there exists $C>0$ such that
\begin{align*}
\varphi(t)^{p(t)}=e^{p(t)\log(\varphi(t))} &\leq e^{|p(t)-p(\theta)|\log(\varphi(t))}e^{p(\theta)\log(\varphi(t))}\\ &\leq  e^{C\log({2\pi}{|t-\theta|})^{-1}}\varphi(t)^{p(\theta)}\\ &\lesssim \varphi(t)^{p(\theta)}.
\end{align*}

\end{proof}

\begin{proof}[Proof of Theorem \ref{thm:bounded_eval_funct}]
Fix $z\in\D$ and let $f\in H^\p(\D)$, then $f\in H^1(\D)$ and we can use the reproducing formula \eqref{eq: reproducing}, and H\"older's inequality to conclude that \[|f(z)\lesssim \norm{f}_{H^\p(\D)}\norm{K_z}_{H^\q(\D)}, \] where $\frac{1}{p(\theta)}+\frac{1}{q(\theta)}=1$ for every $\theta \in [0,2\pi]$.

We will estimate $\norm{K_z}_{H^\q(\D)}$. Let $1/2<r<1$, and define $E_1=\{t\in [0,2\pi]:\varphi(t)\leq 1\}$ and $E_2=[0,2\pi]\setminus E_1$. Then 
\begin{align*}
\int_0^{2\pi}\left((1-|z|)^{1/p(\theta)}|K_z(re^{it})|\right)^{q(t)} \dif t &= \int_0^{2\pi}\left(\varphi(t)\right)^{q(t)} \dif t\\  &= \int_{E_1}\left(\varphi(t)\right)^{q(t)} \dif t +\int_{E_2}\left(\varphi(t)\right)^{q(t)} \dif t\\ &\lesssim 1+ \int_0^{2\pi}\varphi(t)^{q(\theta)}\dif t.
\end{align*}

We now use Forelli-Rudin estimates (See for example, \cite{HKZ}, Theorem 1.7) to conclude that $$\int_0^{2\pi}\varphi(t)^{q(\theta)}\dif t\sim 1.$$

This shows that $\|K_z(r\cdot)\|_{H^\q(\D)}\lesssim \frac{1}{(1-|z|)^{1/p(\theta)}}$ for $1/2<r<1$.

Now, if $0\leq r\leq 1/2$, then  $\varphi(t)\lesssim 1$ for every $t\in[0,2\pi]$ and consequently it also holds that $\|K_z(r\cdot)\|_{H^\q(\D)}\lesssim \frac{1}{(1-|z|)^{1/p(\theta)}}$. Thus, \[\|K_z\|_{H^\q(\D)}\lesssim \frac{1}{(1-|z|)^{1/p(\theta)}}.\]
\end{proof}

The Szesg\"o transformation  $\mathcal{K}:L^p(\TT)\to H^p(\D)$ is defined as \[\mathcal{K}f(z)=\int_\TT \frac{f(\zeta)}{1-\overline{\zeta}z} \dif m(\zeta).\] It is known (see for example \cite{Gar}) that $\mathcal{K}$ is onto for $1<p<\infty$. Moreover, a theorem of Hunt-Muckenhoupt and Wheeden  affirms that $\mathcal{K}$ is bounded on the weighted Lebesgue space $L^\p(\TT,w)$ if and only if $w$ is on the Muckenhoupt class $A_p$. A direct consequence of Rubio de Francia extrapolation (\cite{CUF}, Theorem 5.28)  shows that $\mathcal{K}$ is bounded on $L^\p(\TT)$.

\section{An example}

In this section, we show an example of a variable $\log-$H\"older continuous exponent $\p$ for which $H^\p(\D)$ differs from all classical spaces $H^p(\D)$.
Define $p:[0,2\pi]\to [2,3]$ as

\[p(\theta)=\begin{cases}
3, \text{ if }\theta\in  [0, \frac{\pi}{3}]\cup[\frac{5\pi}{3},2\pi]\\
\cos(\theta)+\frac{5}{2}, \text{ if }  \theta\in ( \frac{\pi}{3},\frac{2\pi}{3})\cup (\frac{4\pi}{3},\frac{5\pi}{3})\\
2, \text{ if }\theta\in [\frac{2\pi}{3},\frac{4\pi}{3}].\\
\end{cases}
\]

Notice that as a consequence of the mean value theorem, if $y\in (\frac{\pi}{3},\frac{\pi}{3}+\frac{1}{2})$, then $\left|\cos(y)-\frac{1}{2}\right|\leq \left|y-\frac{\pi}{3}\right|$. Consequently, if $|x-y|<\frac{1}{2}$ and $x\in [0,\frac{\pi}{3}]$ then \[|p(x)-p(y)|=\left|\frac{1}{2}-\cos(y)\right|\leq |x-y|\leq \frac{1}{\log\left( 1/|x-y| \right)}\] and definition \ref{def:log-Holder} holds for $x\in [0,\frac{\pi}{3}]$. The other cases follow similarly and show that $\p$ is $\log$-H\"older continuous.

Recall that for the functions $f(z)=(1-z)^{-\lambda}$ belong to $H^q(\DD)$ if and only if $\lambda<\frac{1}{q}$.
Let $2<q\leq 3$ and $\varepsilon>0$ such that $\frac{1}{q}<\frac{1}{q}+\varepsilon<\frac{1}{2}$ and define $f_{q,\varepsilon}:\DD\to \CC$ as \[f_{q,\varepsilon}=(1+z)^{-\frac{1}{q}-\varepsilon}.\] Then $f_{q,\varepsilon} \not\in H^q(\DD)$. However,  if $0<r<\frac{1}{2}$, then \[\int_{[0,2\pi]} |f_{q,\varepsilon}(re^{i\theta})|^{p(\theta)}\dif \theta \leq \int_{[0,2\pi]} 2^{p(\theta)}\dif \theta<\infty,\] and if  $\frac{1}{2}<r<1$, then
\begin{align*}
\int_{[0,2\pi]} |f_{q,\varepsilon}(re^{i\theta})|^{p(\theta)}\dif \theta &= \int_{[0,2\pi]\setminus[\frac{2\pi}{3},\frac{4\pi}{3}]}|f_{q,\varepsilon}(re^{i\theta})|^{p(\theta)}\dif \theta+\int_{[\frac{2\pi}{3},\frac{4\pi}{3}]} |f_{q,\varepsilon}(re^{i\theta})|^{p(\theta)}\dif \theta\\
&\leq  \int_{[0,2\pi]} r^{-p(\theta)}\dif \theta + \int_{[0,2\pi]} (1+re^{i\theta})^{-\frac{2}{q}-2\varepsilon}\dif \theta <\infty.
\end{align*}
Therefore $f_{q,\varepsilon}\in H^\p(\DD)$.

On the other hand, let $g_{q,\varepsilon}:\DD\to\CC$ be defined as $g_{q,\varepsilon}(z)=(1-z)^{-\frac{1}{q}-\varepsilon}$, then $g_{q,\varepsilon}\in H^2(\DD)$, however since $\frac{1}{q}+\varepsilon\geq \frac{1}{3}$, we have
\[
\int_{[0,2\pi]} |g_{q,\varepsilon}(re^{i\theta})|^{p(\theta)}\dif \theta\geq \int_{[0,\frac{\pi}{3}]} (1-re^{i\theta})^{-\frac{3}{q}-3\varepsilon}\dif \theta=\infty.
\]
Therefore $g_\varepsilon\not\in H^\p(\DD)$.

This example can be generalized to prove the following.

\begin{pro}
Let $1<q_1<q_2<\infty$, then there exists a $2\pi$-periodic $\log$-H\"older continuous function $p$ such that for every $\theta\in[0,2\pi]$ we have that  $q_1\leq p(\theta)\leq q_2$ and $ H^{\p}(\DD)\neq H^{q}(\DD)$, for every $q_1\leq q \leq q_2$.
\end{pro}

\section{Variable exponent Hardy spaces on the upper half-plane}

In this section, we will consider variable exponents Hardy spaces defined on the upper half-plane $\CC_+=\{z\in\CC: \text{Re}(z)>0\}$. We say that a measurable function $p:\RR_+\to (1,+\infty)$ is $\log$-H\"older continuous if  there exists a positive constant $C_{\log}$ such that
\begin{equation*}
	|p(x)-p(y)|\leqslant \frac{C_{\log}}{\log\left( 1/|x-y| \right)},
	\label{eq:local}
\end{equation*}
for all $x, y \in \RR$., $|x-y|<1/2$. 
We will say that $p$ belongs to the class $LH$ if it is $\log-$H\"older continuous and there exist constants $r_\infty\in\RR$ and $C_\infty>0$ such that for every $x,y\in \RR$,  \[|r(x)-r_\infty|\leq \frac{C_\infty}{\log(e+|x|)}.\] In this section, we will assume that $p$ belongs to the class $LH$. We  denote as $p^+ =    \esss_{x\in \RR}p(x)$ and $ p^-=\essi_{x\in \RR}p(x)$.

 $h(\C_+)$ will denote the space of harmonic functions on $\C_+$. Similarly, we will denote as $H(\C_+)$, the space of analytic functions in $\C_+$. With this notation in hand, we are ready to define the main concept of this section.
\begin{defi}
The variable exponent harmonic Hardy space $h^\p(\C_+)$ is defined as the space of functions $U\in h(\C_+)$ such that
\begin{equation}\label{eq:def_upper}
\sup_{y>0} \int_{-\infty}^\infty |U(x+iy)|^{p(x)}\dif x <\infty.
\end{equation}
Similarly, the variable exponent Hardy space $H^\p(\C_+)$ is defined as the space of functions $U\in H(\C_+)$ that satisfy equation \eqref{eq:def_upper}.
\end{defi}

For $x\in\RR$ and $y>0$, the Poisson kernel of the upper half-plane is defined as \[P_y(x)=\frac{1}{\pi}\frac{y}{x^2+y}.\] Notice that if we define $\Phi:\RR\to\RR$ as \[\Phi(x)=\frac{1}{\pi}\frac{1}{x^2+1}\] and let $\Phi_y(x)=\frac{1}{y}\Phi(\frac{x}{y})$, then $P_y(x)=\Phi_y(x)$ with $\|\Phi_y\|_{L^1(\RR)}=1$ for every $y>0$. Thus, the family $\{\Phi_y\}_{y>0}$ is an approximate identity. We will use several results from \cite{CUF} with this notation that we put together here.

\begin{thm}[\cite{CUF}, Theorems 5.8, 5.9, and 5.11]\label{thm:CU_Poisson}
Suppose $f\in L^\p(\RR)$ and let $y>0$, then \[P_y\ast f(x)=\int_{-\infty}^\infty P_y(x-t)f(t)\dif t<\infty\] and $P_y \ast f(x)\to f(x)$ for almost every $x\in\RR$. If $p^+<\infty$, then the convergence is in measure. Moreover, if additionally $\p$ is in the class $LH$, then there exists $C>0$ such that for every $f\in L^\p(\RR)$, \[\sup_{y>0}\|P_y\ast f\|_{L^\p(\RR)}\leq C\|f\|_{L^\p(\RR)}\] and \[\|P_y\ast f-f\|_{L^\p(\RR)}\to 0, \text{ as } y\to 0.\]
\end{thm}


\begin{thm}\label{thm:bounded_harm}
Let $p:\RR_+\to (1,+\infty)$ be a measurable function in the class $LH$. For each $k>0$ define $H_k=\{z=x+iy\in\CC: y\geq k\}$. If $U\in h^\p(\C_+)$,  then $U$ is bounded on $H_k$.
\end{thm}

\begin{proof}
Let $z_0=a+ib$, with $b>0$ and let $U\in h^\p(\C_+)$. Since $U$ is harmonic, for $r\leq R<b$, we have that 
\[U(z_0)=\frac{1}{2\pi}\int_{-\pi}^\pi U(z_0+re^{it})\dif t,\] consequently 
\begin{align*}
|U(z_0)|\frac{R^2}{2} &\leq \frac{1}{2\pi}\int_0^R \int_{-\pi}^\pi |U(z_0+re^{it})|r\dif r\dif t\\
&\leq  \frac{1}{2\pi}\int_{b-R}^{b+R} \int_{a-R}^{a+R} |U(x+iy)|\dif x\dif y\\
&\leq \frac{1}{\pi}\int_{b-R}^{b+R} \|U_y\|_{L^\p(\RR)}\|\chi_{(a-R,a+R)}\|_{L^\q(\RR)}\dif y
\end{align*}
where $U_y:\RR\to\RR$ denotes the function $U_y(x)=U(x+iy)$ and $\chi_{(a-R,a+R)}$ denotes the characteristic function on the interval $(a-R,a+R)$. It is shown in \cite{rafsam2011}, Lemma 2.5 that \[\|\chi_{(a-R,a+R)}\|_{L^\q(\RR)}\sim (2R)^{1/q(a)}.\] Thus,
\begin{align*}
|U(z_0)|&\lesssim \sup_{y>0}\|U_y\|_{L^\p(\RR)} (R)^{1/q(a)-2}
\end{align*}
which proves the result since $U\in h^\p(\C_+)$.
\end{proof}

Now, lets denote as $h(\overline{\CC_+})$ as the space of harmonic functions on an open half-plane containing $\overline{\CC_+}$. Suppose $U\in h(\CC_+)$ and for $\beta>0$ define $U_\beta(z)= U(z+i\beta)$. Then  $U_\beta\in h(\overline{\CC_+})$ and it is bounded on $\overline{\CC_+}$. The following representation for  $U_\beta$ is proven in \cite{Mas}, Theorem 11.2. For $y>0$, 
\begin{equation}\label{eq:repUbeta}
U_\beta(x+iy)=\frac{1}{\pi}\int_{-\infty}^\infty P_y(x-t)U_\beta(t)\dif t.
\end{equation}

We will use such representation to prove the following.

\begin{thm}
Let $p:\RR_+\to (1,+\infty)$ be a measurable function in the class $LH$. Suppose $U\in h^\p(\CC_+)$. Then there exists a unique $u\in L^\p(\RR)$ such that for every $x\in\RR$ and $y>0$, 
\begin{equation*}\label{eq:them_rep}
U(x+iy)=\frac{1}{\pi}\int_{-\infty}^\infty P_y(x-t)u(t)\dif t.
\end{equation*}
 Moreover, \[\sup_{y>0}\|U_y\|_{L^\p(\RR)}\sim \|u\|_{L^\p(\RR)}.\]
\end{thm}

\begin{proof}
The uniqueness follows from Theorem \ref{thm:CU_Poisson} and the pointwise convergence \[P_y\ast u\to u,~~\text{ as } y\to 0.\] 

To prove the existence, suppose $U\in h^\p(\CC_+)$ and for $\beta>0$ define $U_\beta(z)=U(z+i\beta)$ as before. By Theorem \ref{thm:bounded_harm} $U_\beta$ is bounded on $\overline{\CC_+}$. Also, since $U_\beta\in h(\overline{\CC}_+)$ we have that by equation \eqref{eq:repUbeta} that for every $x\in \RR$, and $y>0$ \[U(x+iy+i\beta)=\frac{1}{\pi}\int_{-\infty}^\infty P_y(x-t)U(t+i\beta)\dif t.\] 

Let $\q$ be the conjugate od $\p$. For each positive integer $n$, let's define $\Lambda_n:L^\q(\RR)\to \RR$ as 
\begin{equation*}
\Lambda_n(f)=\int_{-\infty}^\infty f(x)U_{1/n}(x)\dif x.
\end{equation*}
Notice that by H\"older's inequality, \[|\Lambda_n(f)|\lesssim \|f\|_{L^\q(\RR)}\|U_{1/n}\|_{L^\p(\RR)}\leq\|f\|_{L^\q(\RR)}\sup_{y>0}\|U_{y}\|_{L^\p(\RR)}.\] Hence, \[\|\Lambda_n\|\lesssim \sup{y>0}\|U_{y}\|_{L^\p(\RR)}.\] By Banach-Alaoglu's Theorem, there exists a subsequence $\{\Lambda_{n_k}\}$ that converges to a bounded linear functional $\Lambda\in (L^\q(\RR))^\ast$ in the weak-$\ast$ topology. Using the duality among $L^p(\RR)$ and $L^\q(\RR)$, we can choose $u\in L^\p(\RR)$ such that for every $f\in L^\q(\RR)$, \[\Lambda(f)=\int_{-\infty}^\infty f(t)u(d)\dif t\] and $\|u\|_L^{\p(\RR)}\sim \|\Lambda\|$.

Now define $f_y(t)=P_y(x-t)$ and notice that $f_y\in L^\q(\RR)$. Then 
\begin{align*}
\Lambda(f_y) &= \lim_{k\to\infty} \Lambda_{n_k}(f_y)\\
&=\lim_{k\to\infty} \int_{-\infty}^\infty P_y(x-t)U(t+\frac{i}{n_k})\dif t\\
&= \lim_{k\to\infty} U(x+iy+ \frac{i}{n_k})\\
&=U(x+iy).
\end{align*}
This shows equation \ref{eq:them_rep}. Finally, it follows from Theorem \ref{thm:CU_Poisson}, that \[\sup_{y>0}\|P_y\ast u\|_{L^\p(\RR)}\lesssim \|u\|_{L^\p(\RR)}.\]
\end{proof}

\end{document}